\documentclass[preprint,12pt,authoryear]{elsarticle}

\usepackage{amssymb}
\usepackage{amsthm} 
\usepackage{diagbox} 
\usepackage{multirow}
\usepackage{algorithm}
\usepackage{algorithmic}
\usepackage{color,amssymb,amsmath,amsthm,mathrsfs,natbib}

\usepackage{lineno}

\newtheorem{theorem}{Claim}
\begin{document}

\begin{frontmatter}


\title{a branch-and-price approach for the nurse rostering problem with multiple units}


\author[label1,label2,label3]{Wanzhe Hu}

\address[label1]{School of Economics and Management, Chongqing University of Posts and Telecommunications, Chongqing, China}

\address[label2]{Business School, Sichuan University, Chengdu, Sichuan, China}

\address[label3]{Key Laboratory of Big Data Intelligent Computing, Chongqing University of Posts and Telecommunications, Chongqing, China}

\address[label4]{Department of Industrial and Systems Engineering, Center for Applied Optimization, University of Florida, Gainesville, USA}

\author[label2]{Xiaozhou He}

\author[label2]{Li Luo\corref{cor1}}
\ead{luolicc@163.com}
\cortext[cor1]{Corresponding author}

\author[label4]{Panos M. Pardalos}

\begin{abstract}
In this paper, we study the nurse rostering problem that considers multiple units and many soft time-related constraints. An efficient branch and price solution approach that relies on a fast algorithm to solve the pricing subproblem of the column generation process is presented. For the nurse rostering problem, its pricing subproblem can be formulated as a shortest path problem with resource constraints, which has been the backbone of several solutions for several classical problems like vehicle routing problems. However, approaches that perform well on these problems cannot be used since most constraints in the nurse rostering problem are soft. Based on ideas borrowed from global constraints in constraint programming to model rostering problems, an efficient dynamic programming algorithm with novel label definitions and dominating rules specific to soft time-related constraints is proposed. 
In addition, several acceleration strategies are employed to improve the branch and price algorithm. Computational results on instances of different sizes indicate that the proposed algorithm is a promising solution for the nurse rostering problem with multiple units. 
\end{abstract}




\end{frontmatter}

\section{Introduction}
The Nurse Rostering Problem (NRP), also known as the nurse scheduling problem, aims to construct a high-quality roster for a set of nurses over a given scheduling horizon. 
A roster is a collection of individual schedules for nurses, each of which specifies an ordered list of shift sequences (e.g., early, late, or night) and day-off periods \citep{brucker2010shift}.
A number of studies on the implications of poor nurse scheduling show that poor scheduling is closely related to poor patient care, poor nurse morale, reduced patient satisfaction, and ultimately poor hospital performance \citep{clark2015rescheduling}.
The real-world NRP frequently involves a large number of intricate constraints, and many of its variants are classified as NP-hard \citep{brucker2011personnel, smet2016polynomially}.
Due to the importance and complexity, NRPs have been extensively studied in the past several decades, resulting in a considerable amount of literature on the modeling and solving methods to various NRPs. For literature reviews the reader is referred to \citet{burke2004state}, \citet{cheang2003nurse}, and \citet{ngoo2022survey}.

The First and Second International Nurse Rostering Competitions (INRC-I and INRC-II), which were held in 2010 and 2015 respectively \citep{ceschia2019second}, provided two specific problem formulations and a large number of available instances, which makes it easier for researchers to compare their algorithms with others. Consequently, most papers published recently have sought to develop better solutions to these problems based on the corresponding instances. For example, \citet{santos2016integer}, \citet{zheng2017simple}, and  \citet{rahimian2017hybrid} proposed new models or algorithms to solve problems similar to the one for the INRC-I. \citet{mischek2019integer}, \citet{legrain2020online}, and \citet{kheiri2021hyper} reported their solutions to the multi-stage NRP defined by the INRC-II. Additionally, the benchmark instances presented by \citet{curtois2014computational} have also been widely used to test algorithms. \citet{knust2019simulated}, \citet{turhan2020hybrid}, 
and \citet{chen2022neural} are some of the researchers who have used these instances. 

All these problems and instances mentioned above deal with general NRPs that address the work assignments of nurses within one unit, so the decision variables of their models consist in assigning nurses to different shifts each day. 
However, there are some situations where multiple units are involved and additional unit assignment decisions must be made.    Centralized scheduling, where the units with similar nurse requirements are scheduled centrally to make better use of nurse resources, is one such example \citep{burke2004state}. Another is the rostering problem of float nurses, who do not have a home unit and work in different units as needed
\citep{wright2010strategies, maenhout2013integrated}. In these cases, nursing managers are required to deal with the Nurse Rostering Problem with Multiple Units (NRPMU), which takes into account additional unit allocation decisions.

Most NRPs are characterized by a large number of time-related constraints for individual schedules. Representative constraints include limiting the maximum and minimum numbers of a specific roster item (e.g., assignments or days-off) within the scheduling horizon and restricting the maximum and minimum numbers of consecutive occurrences of specific roster items, which are referred to as “ranged” counter constraints and series constraints, respectively \citep{smet2016polynomially}. Moreover, these constraints can be hard or soft ones. To generate a feasible roster, hard constraints must be satisfied while soft constraints can be violated
with penalties. 
To our knowledge, only a few papers deal with the NRPMU with counter and series constraints. 
\citet{wright2013centralized}, \citet{leksakul2014nurse}, and \citet{fugener2018mid} proposed models involving multiple units, yet only a minority of time-related constraints 
are considered. \citet{turhan2022mat} presented a model that considers multiple units and most time-related constraints. They treated most time-related constraints as hard constraints. \citet{maenhout2013integrated} proposed an integrated nurse staffing and scheduling method which involves multiple units and a number of nurse-specific constraints and objectives. In contrast, this paper focuses on the scheduling phase. The NRPMU under study considers a range of common soft counter and series constraints as well as novel soft constraints regarding unit allocation decisions.


Although a variety of methods have been proposed to solve NRPs, such as exact methods, heuristics, and hybridizations of these techniques, branch and price ($B\&P$) has been proven to be very competitive \citep{burke2014new}. It is well-known that $B\&P$ is a branch and bound method in which Column Generation (CG) techniques are used to get the lower bound at each node of the search tree. CG considers two problems: the master problem, which is a linear programming problem, and the pricing subproblem, which in NRPs is generally formulated as a Shortest Path Problem with Resource Constraints (SPPRC), an NP-hard combinatorial optimization problem \citep{irnich2005shortest}. Apparently, the efficiency of the algorithm to solve the pricing subproblem is at the core of the CG and $B\&P$. 
Indeed, the dynamic programming (DP) method has been extensively applied to solve the SPPRC for the vehicle routing problems (VRPs), and several  acceleration techniques are proposed \citep{costa2019exact}. 
Nevertheless, due to the ranged counter and series constraints, existing DP algorithms for VRPs cannot be directly used to solve the SPPRC for NRPs.

Multiple approaches based on $B\&P$ or CG to solve NRPs have been reported during the last several decades. Some methods were proposed to tackle the pricing subproblem.  One of the earliest attempts to solve NSPs using $B\&P$ was made in 1998 \citep{jaumard1998generalized}.
They presented their $B\&P$ method to solve the NRP that takes demand coverage and nurse preferences into account. They formulated the pricing subproblem as a SPPRC and proposed an efficient two-phase algorithm. 
%
\citet{bard2005preference} reported a CG-based method to address the NRP subject to demand requirements and personel considerations, which they called pereference scheduling. Columns were generated using a heuristic.
To reduce the instability, 
\citet{purnomo2007cyclic} studied the cyclic preference scheduling problem and developed a $B\&P$ algorithm. The pricing subproblem was solved by CPLEX as an integer programming problem.
\citet{maenhout2010branching} presented an exact $B\&P$ algorithm for solving the NRP with multiple objectives. A two phase approach was implemented to solve the pricing subproblem. If the heuristic in the first phase failed, an exact dynamic programming method was used.
\citet{strandmark2020first} developed a CG based heuristic that is able to produce good solutions quickly for large NRP instances. The pricing subproblem was modelled as a SPPRC and solved using a DP-based heuristic.
\citet{guo2022column} investigated the NRP that accommodates overtime. A CG-based heuristic with stabilization was proposed, and the pricing subproblem was solved using CPLEX.
\color{black}

The approaches mentioned above are mainly developed to deal with NRPs with hard counter or series constraints. In contrast, there are only a few CG-based solutions published to address NRPs with soft ones. \citet{he2012constraint} formulated the pricing subproblem with soft constraints as a constraint programming model, whose main task was to generate good individual schedules rather than optimal ones. 
\citet{burke2014new} proposed a DP algorithm to solve the SPPRC for the NRP with soft counter and series constraints. The dominance rules are adapted directly from the methods applied for VRPs. According to our analysis given in section \ref{sec_dominance}, these dominance rules cannot deal effectively
with the soft counter and series constraint, which leads to a great number of labels to be extended, thereby degrading the algorithm performance.
It is worth noting that Omer and Legrain \citep{legrain2020rotation, omer2023dedicated} have reported two $B\&P$ methods to solve NRPs. Their latest research \citep{omer2023dedicated}, a
preprint article, is online recently while we are conducting this research. Dedicated algorithms for the SPPRC with soft and hard constraints involved in the NRPs are presented. Although our method shares some similarities with theirs, the resulting dominance rules are indeed different. In addition, we provide some different ideas to define and update labels. Main differences between their method and ours are given in section \ref{subsec_differ}.

The aim of this paper is to address the NRPMU involving several soft counter and series constraints, which are often encountered in real-world NRPs. 
%
The main contributions of this paper are twofold.
First, we investigate the NRPMU with a number of soft time-related constraints, which has been dealt with by few papers. Objectives and constraints specific to NRPMU are taken into account to match nurse skills with units and
balance the distribution of the nursing workload for each unit among available nurses.
Second, we present a $B\&P$ approach to solve our NRPMU. An efficient yet optimal DP algorithm is developed to tackle the SPPRC with soft counter and series constraints. New label definitions and dominance rules are introduced. These methods can be generalized to handle general NRPs. In addition, we propose a tailored accelerating strategy for the NRPMU. In contrast to traditional methods that use dominance rules within one node, we apply our dominance rules across multiple nodes.

The remainder of this paper is organized as follows.
Section \ref{problem formulation} formally describes  the master problem and pricing subproblem for the NRPMU under consideration. Section \ref{sec_solution} presents the proposed $B\&P$ algorithm, with a particular emphasis on describing the label definitions and dominating rules we introduced. In section \ref{sec_experiments}, computational experiments are conducted on a number of instances. Section \ref{sec_conclusion} concludes this paper. 

\section{Problem formulation}
\label{problem formulation}


Given a set of nurses and the coverage requirements for each day, unit, and shift during a scheduling period, the NRPMU is defined as the process of determining a set of individual schedules according to a set of objectives and constraints. 
An individual schedule specifies one nurse’s rest days and work days. For each work day, the nurse is assigned to a unit (workplace) and a shift (working period). 
The collection of individual schedules for all nurses is called a roster. 


The NRPMU involves many constraints arising from work regulations and nurses’ preferences. 
Most constraints are soft and the objective is to find a feasible roster with the minimum penalty. Since multiple units are involved, the NRPMU is characterized by additional decisions and constraints.  
First, it is essential to match nurses with units that require different skills. For each unit, the nurses are divided into three categories: nurses with preferred skills, with required skills, and without required skills. Nurses without required skills cannot be assigned to corresponding units, 
and assigning nurses with required but not preferred skills results in penalties.
Furthermore, to balance the distribution of the nursing workload for each unit among nurses, soft constraints that limit the minimum and maximum number of days worked in one unit are introduced for each nurse. 

Here, we first detail the NRPMU’s constraints and then present the mathematical formulations in the following subsections. We assume that the scheduling period always consists of a whole number of weeks and starts on a Monday. 

Hard constraints are:
\begin{enumerate}
\item One nurse can be assigned a maximum of one unit and one shift per day.
\item Forbidden shift rotations. A shift type cannot be followed by some shift types on the next day. For our instances, the following shift rotations are not allowed, namely, Late-Early, Night-Early, and Night-Late.
\item Required skill. One cannot assign a nurse to a unit for which he/she does not have corresponding required/preferred skills.

\end{enumerate}

Soft constraints are:
\begin{enumerate}

\item Maximum and minimum number of working days for each nurse.
\item Maximum and minimum number of days worked in one unit for each nurse.
\item Maximum number of working weekends for each nurse.
\item Maximum and minimum number of consecutive working days for each nurse. 
\item Maximum and minimum number of consecutive rest days for each nurse.
\item Day on/off request. Requests by the nurses to work or not to work on specific days.
\item Shift on/off request. Requests by the nurses to work or not to work on specific shifts of certain days.
\item Preferred skill. Assigning nurses without preferred skills to corresponding units leads to penalty.
\item Cover requirements. If the number of nurses assigned to a shift of a unit on one day is less than its requirement, penalty occurs.
\end{enumerate}
\subsection{The master problem}
This section introduces the master problem that is involved in the CG process to solve the NRPMU. Notations used for the model is given as follows.


\noindent\textbf{Sets}:
\begin{description}
    \item[$N$] Set of nurses, indexed by $n$
    \item[$D$] Set of days, indexed by $d$
    \item[$U$] Set of units, indexed by $u$
    \item[$S$] Set of shifts, indexed by $s$
    \item[$L_n$] Set of feasible individual schedules for nurse $n$, indexed by $l$
\end{description}

\noindent\textbf{Parameters}:
\begin{description}
    \item[$c_{nl}$] Sum of the penalties of an individual schedule $l \in L_n$, which results from violating soft constraints
    \item[$r_{dus}$] the number of nurses required in shift $s$ and unit $u$ on day $d$
    \item[$a_{nldus}$] 1 if nurse $n$ is assigned to shift $s$ and unit $u$ on day $d$ in individual schedule $l$, 0 otherwise
    \item[$p_{dus}^{under}$] Penalty for understaffing in shift $s$ and unit $u$ on day $d$
\end{description}

\noindent\textbf{Decision Variables}:
\begin{description}
    \item[$x_{nl}$] 1 if individual schedule $l$ is selected by nurse $n$, 0 otherwise
    \item[$v_{dus}$] The number of nurses understaffed in shift $s$ and unit $u$ on day $d$
\end{description}

The NRPMU is formulated as the following integer linear programming model.

\begin{align}
\label{obj1} \text{(MP)} \quad \min \quad  & \sum_{n \in N} \sum_{l \in L_n} x_{nl} \cdot c_{nl} + \sum_{d \in D} \sum_{u \in U} \sum_{s \in S} p_{dus}^{under} \cdot v_{dus} \\
\label{cons1} \text{s.t.}  \quad  & \sum_{n \in N} \sum_{l \in L_n} x_{nl} \cdot a_{nldus} + v_{dus} \ge r_{dus}, \quad\forall d \in D, u \in U, s \in S \\
\label{cons2} & \sum_{l \in L_n} x_{nl} = 1, \quad\forall n \in N \\
\label{cons3} & x_{nl} \in \{0,1\}, \quad\forall n \in N, l \in L_n \\
\label{cons4} & v_{dus} \in \mathbb{N}, \quad\forall d \in D, u \in U, s \in S
\end{align}

The objective is to minimize the sum of penalties caused by all nurses' schedules and understaffing in each shift and unit during the scheduling period. Constraints (\ref{cons1}) identify the number of nurses below the cover requirement $r_{dus}$ for each shift and unit on each day. Constraints (\ref{cons2}) enforce that one individual schedule must be selected for each nurse. Constraints (\ref{cons3}) and (\ref{cons4}) define the decision variables. Note that the variables $x_{nl}$ are defined as being integer so as to avoid constraints $x_{nl} \le 1$ in the linear relaxation of (\ref{obj1}-\ref{cons4}). Moreover, due to constraints (\ref{cons2}), it is obvious that any solution with $x_{nl} > 1$ would be unfeasible.

Let us call the linear relaxation of (\ref{obj1}-\ref{cons4}) the Master Problem (MP), and it has been proven that the MP can provide a good lower bound for the branch and bound algorithm. However, since the size of each set $L_n$ grows exponentially with the number of shifts, units, and days, it is not tractable to consider all possible individual schedules in the model explicitly. In practice, we work with the MP with a small subset $ L^{'}_n  \in L_n$  for each nurse, which is called the Restricted Master Problem (RMP). 

Let $\lambda_{dus}$ and $\lambda_n$ be the dual variables for the constraints (\ref{cons1}) and (\ref{cons2}), respectively.
Provided that we have a RMP with $\lambda_{dus}^{'}$ and $\lambda_n^{'}$ being the optimal solutions to its dual program, the reduced cost of individual schedule $l$ for nurse $n$ in the simplex method for MP can be denoted as 

$rc_{nl} = c_{nl} - \sum_{d \in D} \sum_{u \in U} \sum_{s \in S} a_{nldus} \cdot \lambda_{dus} - \lambda_{n}$. If individual schedules with negative $rc_{nl}$ can be found, we add them, so-called columns, to the RMP and re-optimize the RMP. Otherwise, the MP has been solved optimally. This iterative process is known as the CG method, and the problem used to generate individual schedules with negative reduced cost is called the pricing subproblem.

\subsection{The pricing subproblem}
Since nurses may have different contract requirements and individual requests, they are treated to be heterogeneous in the model. Consequently, the procedure to generate columns is composed of $|N|$ pricing subproblems, each of which aims to find the individual schedule $l$ with the minimum reduced cost $\overline{rc}_{nl}$ for nurse $n$. It is clear that RMP only takes into account cover requirements explicitly. In contrast, the other constraints are dealt  with implicitly by the definition of feasible individual schedules, which are generated by pricing subproblems. In order to describe these constraints precisely, we present an integer programming model to formulate the pricing subproblem for nurse $n$, which is based on the modeling method introduced by \citet{santos2016integer}. The notations not mentioned above are defined as follows:

\noindent\textbf{Parameters}:
\begin{description}
  \item[$p^{dnu}$] penalty for violating the constraint limiting the maximum number of working days 
  \item[$p^{dnl}$] penalty for violating the constraint limiting the minimum number of working days 
  \item[$p^{udnu}_u$] penalty for violating the constraint limiting the maximum number of days worked in unit $u$
  \item[$p^{udnl}_u$] penalty for violating the constraint limiting the minimum number of days worked in unit $u$
  \item[$p^{wn}$] penalty for violating the constraint limiting the maximum number of working weekends
  \item[$\Pi$] set of all ordered pairs $(d_1,d_2) \in D \times D$ with $d_1 \leq d_2$ 
  \item[$p_{d_1 d_2}^{don}$] pre-computed penalty for consecutive work from day $d_1$ to day $d_2$
  \item[$p_{d_1 d_2}^{doff}$] pre-computed penalty for consecutive rest from day $d_1$ to day $d_2$
  \item[$p_d^{don}$] penalty for the request to work on day $d$
  \item[$p_d^{doff}$] penalty for the request not to work on day $d$
  \item[$p_{ds}^{son}$] penalty for the request to work on shift $s$ of day $d$ 
  \item[$p_{ds}^{soff}$] penalty for the request not to work on shift $s$ of day $d$
  \item[$p_{du}^{np}$] penalty if unit $u$ is assigned to the nurse without corresponding preferred skills.
  \item[$\eta_u$] 1 if the nurse has preferred skills for unit $u$ and 0 otherwise
  \item[$W$] set of weekends in the scheduling period
  \item[$D_i$] set of days in the $i$-th weekend, $i \in W$
  \item[$\alpha_n^l$, $\alpha_n^u$] the minimum and maximum number of working days for nurse $n$
  \item[$\alpha_n^{ul}$, $\alpha_n^{uu}$] the minimum and maximum number of days worked in unit $u$ for nurse $n$
  \item[$\beta_n$] the maximum number of working weekends for nurse $n$
  \item[$S^{forbidden}_s$] set of shifts that cannot follow shift $s$
  \item[$N_u^{require}$] set of nurses with required or preferred skills for unit $u$
\end{description}

\noindent\textbf{Decision Variables}:
\begin{description}
  \item[$q^{dnu}$] Total number of working days above the maximum number limit 
  \item[$q^{dnl}$] Total number of working days below  the minimum number limit
  \item[$q^{udnu}_u$] Total number of working days above the maximum number limit in unit $u$ 
  \item[$q^{udnl}_u$] Total number of working days below the minimum number limit in unit $u$
  \item[$f$] Total number of working weekends above the maximum number limit
  \item[$w_{d_1 d_2}$] 1 if the nurse works  from day $d_1$ until day $d_2$ and 0 otherwise
  \item[$r_{d_1 d_2}$] 1 if the nurse rests from day $d_1$ until day $d_2$ and 0 otherwise
  \item[$\theta_{dus}$] 1 if the nurse is assigned to shift $s$ and unit $u$ on day $d$ and 0 otherwise
\end{description}

The integer programming model for the pricing subproblem for nurse $n$:
\begin{equation}
    \overline{rc}_{nl} = min \quad F_r + F_o + F_d + F_s + F_p - \sum_{d \in D} \sum_{u \in U} \sum_{s \in S} \theta_{dus} \cdot \lambda_{dus} - \lambda_{n} \label{spObj}
\end{equation}
\begin{equation}
    F_r = p^{dnu} \cdot q^{dnu} + p^{dnl} \cdot q^{dnl} + p^{udnu}_u \cdot q^{udnu}_u + p^{udnl}_u \cdot q^{udnl}_u + p^{wn} \cdot f
\end{equation}
\begin{equation}
    F_o = \sum_{d_1 d_2 \in \Pi} (p_{d_1 d_2}^{don} \cdot w_{d_1 d_2} + p_{d_1 d_2}^{doff} \cdot r_{d_1 d_2})
\end{equation}
\begin{equation}
    F_d = \sum_{d \in D} p_d^{don} \cdot (1-\sum_{s \in S} \sum_{u \in U} \theta_{dus}) + 
     \sum_{d \in D} p_d^{doff} \cdot \sum_{s \in S} \sum_{u \in U} \theta_{dus}
\end{equation}
\begin{equation}
    F_s = \sum_{d \in D} \sum_{s \in S} p_{ds}^{son} \cdot (1- \sum_{u \in U} \theta_{dus}) + \sum_{d \in D} \sum_{s \in S} p_{ds}^{soff} \cdot  \sum_{u \in U} \theta_{dus}
\end{equation}
\begin{equation}
    F_p = \sum_{d \in D} \sum_{u \in U} p_{du}^{np} \cdot (1-\eta_u) \cdot \sum_{s \in S} \theta_{dus}
\end{equation}
subject to
\begin{equation}
    \sum_{u \in U} \sum_{s \in S} \theta_{dus} \le 1, \forall d \in D
    \label{spcons1}
\end{equation}
\begin{equation}
    \sum_{u \in U} \sum_{s \in S} \theta_{dus} = \sum_{[d_1,d_2] \in \Pi: d \in [d_1,d_2]} w_{d_1 d_2}, \forall d \in D
    \label{spcons2}
\end{equation}
\begin{equation}
     \sum_{d_1 \in \{1,\cdots,d\}} w_{d_1 d} + \sum_{d_2 \in D: d_2 \ge d+1} w_{(d+1)d_2} \le 1, \forall d \in D
     \label{spcons3}
\end{equation}
\begin{equation}
    \sum_{u \in U} \sum_{s \in S} \theta_{dus} = 1 - \sum_{[d_1,d_2] \in \Pi: d \in [d_1,d_2]} r_{d_1 d_2}, \forall d \in D
    \label{spcons4}
\end{equation}
\begin{equation}
     \sum_{d_1 \in \{1,\cdots,d\}} r_{d_1 d} + \sum_{d_2 \in D: d_2 \ge d+1} r_{(d+1)d_2} \le 1, \forall d \in D
     \label{spcons42}
\end{equation}
\begin{equation}
    \alpha_n^l - q^{dnl} \le \sum_{d \in D} \sum_{u \in U} \sum_{s \in S} \theta_{dus} \le \alpha_n^u + q^{dnu}
    \label{spcons5}
\end{equation}
\begin{equation}
    \alpha_{nu}^{ul} - q^{udnl}_u \le \sum_{d \in D} \sum_{s \in S} \theta_{dus} \le \alpha_{nu}^{uu} + q^{udnu}_u, \forall u \in U
    \label{spcons6}
\end{equation}
\begin{equation}
    y_i \ge \sum_{u \in U} \sum_{s \in S} \theta_{dus}, \forall i \in W, d \in D_i
    \label{spcons7}
\end{equation}
\begin{equation}
    y_i \le \sum_{d \in D_i} \sum_{u \in U} \sum_{s \in S} \theta_{dus}, \forall i \in W
    \label{spcons8}
\end{equation}
\begin{equation}
    \sum_{i \in W} y_i \le f + \beta_n
    \label{spcons9}
\end{equation}
\begin{equation}
    \sum_{u \in U} \theta_{dus} + \sum_{u \in U} \theta_{(d+1)u\underline{s}} \le 1, \forall d \in D, s \in S, \underline{s} \in S^{forbidden}_s 
    \label{spcons10}
\end{equation}
\begin{equation}
    \theta_{dus} = 0, \forall d \in D, u \in U, s \in S \quad if \quad n \notin N_u^{require}
    \label{spcons11}
\end{equation}
\begin{equation}
    q^{dnu}, q^{dnl}, f \in \mathbb{N}
     \label{spcons12}
\end{equation}
\begin{equation}
    q^{udnu}_u,q^{udnl}_u \in \mathbb{N}, \forall u \in U
\end{equation}
\begin{equation}
    w_{d_1 d_2}, r_{d_1 d_2} \in \{0,1\}, \forall (d_1, d_2) \in \Pi
\end{equation}
\begin{equation}
    \theta_{dus} \in \{0,1\}, \forall d \in D, u \in U, s \in S  \label{spcons13}
\end{equation}

In the objective (\ref{spObj}), the total penalty $c_{nl}$ for individual schedule $l$ is composed of $F_r$, $F_o$, $F_d$, $F_s$, and $F_p$, which results from violating soft constraints 1-3, 4-5, 6, 7, and 8, respectively. Constraints (\ref{spcons1}), (\ref{spcons10}), and (\ref{spcons11}) ensure that the hard constraints must be satisfied. Constraints from (\ref{spcons2}) to (\ref{spcons42}) link variables $\theta_{dus}$ with $w_{d_1 d_2}$ and $r_{d_1 d_2}$. Constraints (\ref{spcons5}) model the maximum and minimum number of working days in the scheduling period. Constraints (\ref{spcons6}) limit the maximum and minimum number of days worked for each unit. Constraints (\ref{spcons7}) and (\ref{spcons8}) link variables $\theta_{dus}$ with $y_i$ and guarantee that a working weekend is defined as a weekend with at least one working day. Constraints (\ref{spcons9}) restrict the maximum number of working weekends. The variables are defined by constraints from (\ref{spcons12}) to (\ref{spcons13}).

\section{Solution approach}
\label{sec_solution}
Although the CG method is able to solve the MP optimally, it is not guaranteed that the resulting solution is integral. The $B\&P$ is a hybrid of branch and bound and CG methods, where branching occurs when the optimal solution to the MP is not integral. For the CG method is employed to solve the MP with additional branching constraints at each node, the performance of the $B\&P$ depends heavily on the CG solutions. 



As stated above, the performance of the algorithm to solve pricing subproblems is crucially important.
Based on the work by \citet{burke2014new} and constraint programming solutions for rostering problems \citep{demassey2006cost, metivier2009solving}, an exact yet efficient DP algorithm for the NRPMU is presented in this paper. In the following, the label definitions and dominance rules are described in detail. 

\subsection{The label definition}
\label{sec_LableDefinition}
The SPPRC for the NRPMU is defined on a directed graph $G = (V, A)$, where $A$ is the set of arcs and $V = \{ v_0, ..., v_n, v_{n+1}\} $ is the set of nodes. Vertexes $v_0$ and $v_{n+1}$ are the source and sink nodes, respectively. Each of the other nodes represents a working day in shift $s$ and unit $u$ on day $d$ or a day off. Assuming that there are the same shift types in each unit, $n$ is equal to $|D| \times (|U| \times |S| + 1) $. Figure \ref{graph} gives an example graph for a NRPMU instance with two units, which are indexed by 1 and 2 respectively, and two shifts (Day and Night). Note that $A$ can only contain three types of arcs: arcs from the source node to nodes on the first day, arcs from nodes on the last day to the sink node, and arcs from nodes on one day to nodes on the next day. A path from the source node to the sink node represents an individual schedule for one nurse. 

\begin{figure}[htbp]
\centering
\includegraphics[width = .8\textwidth]{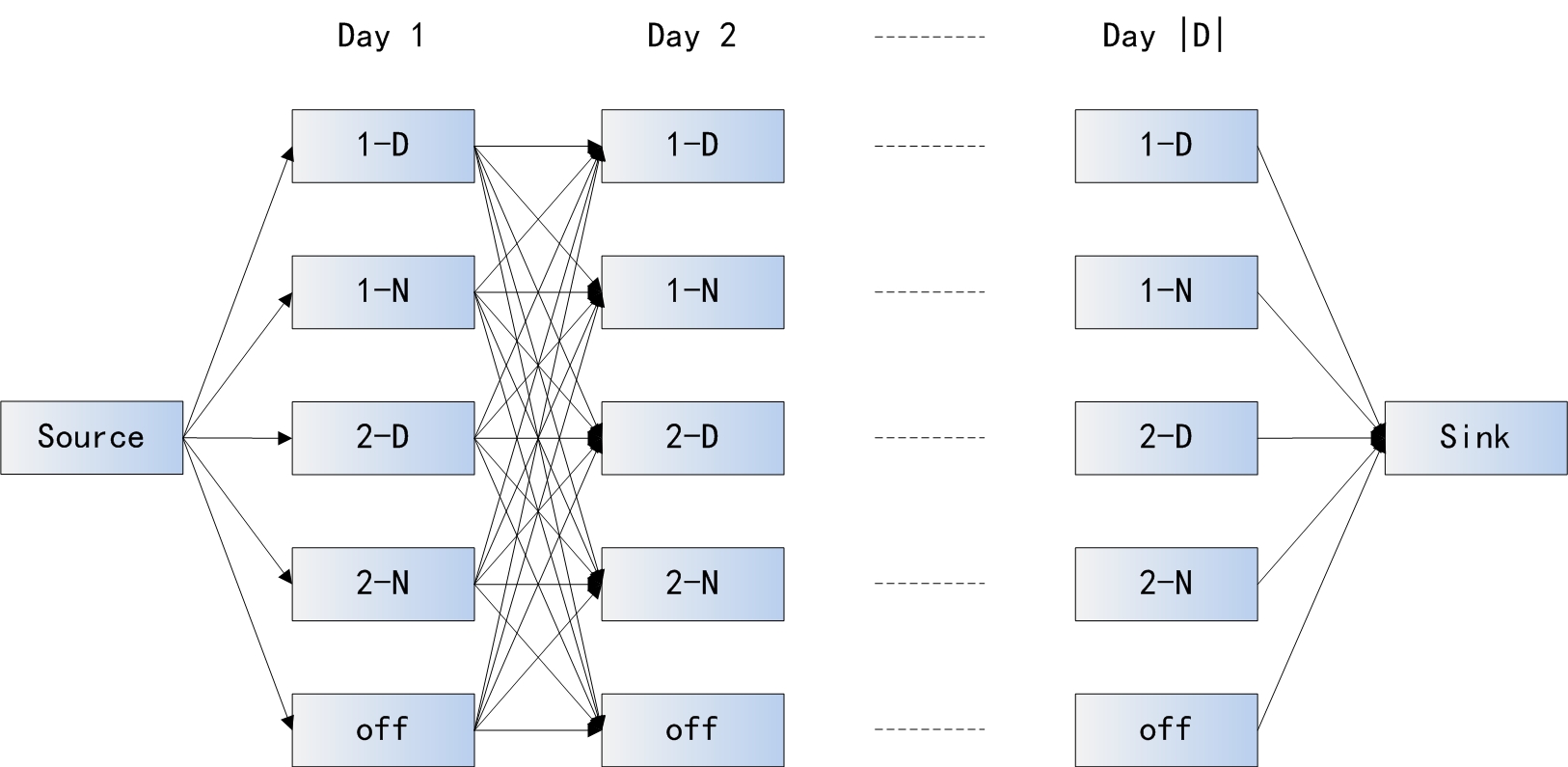}
\caption{Example graph for the SPPRC for a NRPMU instance}
\label{graph}
\end{figure}

For the classical SPPRC, with each arc $(v_i, v_j) \in A$ are associated a cost $c_{ij}$ and a set of values denoting resource consumption $\{ d_{ij}^r| r \in RS \}$, where $RS$ is the set of resources. Assuming that only capacity constraints are involved, there is a maximum consumption limit for each resource. Given a partial path $P_i$ from the source node to a node $v_i$, the cost $C_i$ and the quantity $R_i^r$ of a resource $r$ consumed by the path are calculated by $C_i = \sum_{(v_i, v_j) \in P_i} c_{ij}$ and $R_i^r = \sum_{(v_i, v_j) \in P_i} d_{ij}^r$, respectively. The objective is to find a minimum cost path from the source node to the sink node that does not exceed all maximum resource consumption limits. Dynamic programming algorithms for the SPPRC construct paths by iteratively extending a partial path at one node into all its successor nodes. In the beginning, there is only one partial path that solely consists of the source node. The algorithm terminates once no new paths are created. Dominance rules are introduced to identify and discard non-useful paths so that the number of paths to be extended is as small as possible. In fact, the label is introduced in the algorithm to stand for a path and record its cost and resource consumption. As a consequence, the dynamic programming algorithm for the SPPRC is also called the labeling algorithm.

Let us recall that the objective of the pricing subproblem is to find the individual schedule with the minimum reduced cost $\overline{rc}_{nl} = c_{nl} - \sum_{d \in D} \sum_{u \in U} \sum_{s \in S} a_{nldus} \cdot \lambda_{dus} - \lambda_{n}$ for each nurse. Since $\lambda_{n}$ is a constant for nurse $n$, it is equivalent to finding the complete path with the minimum $c_{nl} - \sum_{d \in D} \sum_{u \in U} \sum_{s \in S} a_{nldus} \cdot \lambda_{dus}$ in the graph constructed for the NRPMU. Hence, we define as $c_{nl} - \sum_{d \in D} \sum_{u \in U} \sum_{s \in S} a_{nldus} \cdot \lambda_{dus}$ the cost of one complete path in the SPPRC we created for the NRPMU. The dual variables $ \lambda_{dus} $ can be easily handled by adding them to the costs of corresponding arcs. The difficulty lies in computing $c_{nl}$, which results from violations of soft constraints.

There are eleven constraints involved in the pricing subproblem for each nurse. Some of these constraints can be tackled by the SPPRC easily. Since a path contains only one node on the same day, hard constraints 1 are respected automatically. Hard constraints 2 and 3 can be satisfied by excluding corresponding arcs from the arc set $A$. Soft constraints 6 to 8 can be dealt with by associating suitable costs with arcs toward the nodes denoting day/shift on/off or units for which the nurse does not have preferred skills. In contrast, it is hard to address the other soft constraints, which are special ones arising out of the NRPMU or other NRP related problems. 

According to \citet{smet2016polynomially}, these constraints are classified into ranged counter constraints (soft constraints 1-3) and series constraints (soft constraints 4 and 5).
\citet{burke2014new} presents an efficient $B\&P$ algorithm for the NRP, where they treat the counter constraints as general constraints called "Workload", and introduce regular expression constraints to describe the series constraints. Regular expressions are commonly used in Computer Science to specify text patterns to be matched. If we treat the individual schedule as the search text, the series constraint can be handled by forbidding or limiting the occurrence of specified patterns. Several examples are given to illustrate how to express series constraints by regular expression constraints. Pricing subproblems are solved by a dynamic programming algorithm where labels they called partial patterns are used to count the number of occurrences of the item limited by counter constraints, such as working days and rest days, and the number of matches of each regular expression. Indeed, it is easy to notice that how the item representing one counter constraint is counted is almost the same as the resource in the classical SPPRC.

In this paper, we also introduce one resource (item) for each counter constraint in the label while dealing with series constraints in a different way. Given that \citet{burke2004state} describes most concepts related to labels in words, it is not easy to understand and implement their approach to handle labels. To the best of our knowledge, solving rostering problems based on regular expressions appears first in the field of constraint programming. \citet{pesant2004regular} introduces a new global constraint named $regular$ to restrict the values taken by a fixed-length sequence of variables $\bf x$. Each regular constraint corresponds to a deterministic finite automaton (DFA). A DFA can be described by a 5-tuple which includes $Q$ as a finite set of states, $ \Sigma$ as an alphabet, $ \delta $ as a partial transition function mapping $Q \times \Sigma$ to $Q$, $q_0$ as the initial state, and $F$ as a subset of $Q$ representing the final or accepting states. When a string is given as input, the automaton begins at the initial state $q_0$ and processes the string symbol by symbol. At each step, the transition function $\delta$ is applied to update the current state. The string is deemed accepted if and only if the last state reached belongs to $F$. If variables $\bf x$ are constrained by a regular constraint with a DFA $M$, any sequence of values taken by $\bf x$ should be accepted by $M$. Following this research, \citet{demassey2006cost} presents a cost-version of $regular$ constraints. \citet{metivier2009solving} proposes a constraint programming model to solve NRPs with $cost-regular$ constraints and gives several examples to show how to construct DFAs for soft NRP constraints. 

Here, we borrow some ideas from these research to define our labels. For each series constraint, we first construct a DFA with cost and then add a variable in the label to represent its state at one node. Let us take soft constraint 4 as an example to detail the procedure to construct DFAs. Assume that the minimum and maximum number of consecutive working days is set to $d_{min}$ and $d_{max}$ respectively. Let the state $S_i$ denote that one nurse has worked consecutively for $i$ days. There are a total of $d_{max} + 1$ states $Q = \{ S_0, S_1, ..., S_{d_{max}} \}$ where $q_0 = S_0$ is the initial state and all states are accepted, that is, $F = Q$. The alphabet $\Sigma$ consists of two values $\{ 1, 0\}$, where 1 represents a working day and 0 otherwise. The transition function $\delta$ is defined by the state transition table (Table \ref{tb_DFA}). The leftmost column indicates current states, and the other two columns give the outputs, which consist of the next states and costs associated with the state transitions, after inputting 1 and 0 respectively. There is also a graphical representation for a DFA. Figure \ref{DFA} gives an example DFA for soft constraint 4, where arcs denote state transitions. It is obvious that the DFA for soft constraint 5 can be created in the same way.  

\begin{table}[htbp]
\centering
\begin{tabular}{ccc}
\hline
Current state   & 1       & 0             \\ \hline
$S_0$ & $S_1$/0    & $S_0$/0          \\
... & ...    & ...          \\
$S_i$ & $S_{i+1}$/0    & $S_0$/$(d_{min}-i) \times C_{min}$ \\
... & ...    & ...          \\
$S_{d_{min}}$ & $S_{d_{min}+1}$/0    & $S_0$/0          \\

... & ...    & ...          \\
$S_{d_{max}-1}$ & $S_{d_{max}}$/0    & $S_0$/0          \\
$S_{d_{max}}$ & $S_{d_{max}}$/$C_{max}$ & $S_0$/0          \\ \hline
\end{tabular}
\caption{The example state transition table}
\label{tb_DFA}
\end{table}

\begin{figure}[htbp]
\centering
\includegraphics[width = .8\textwidth]{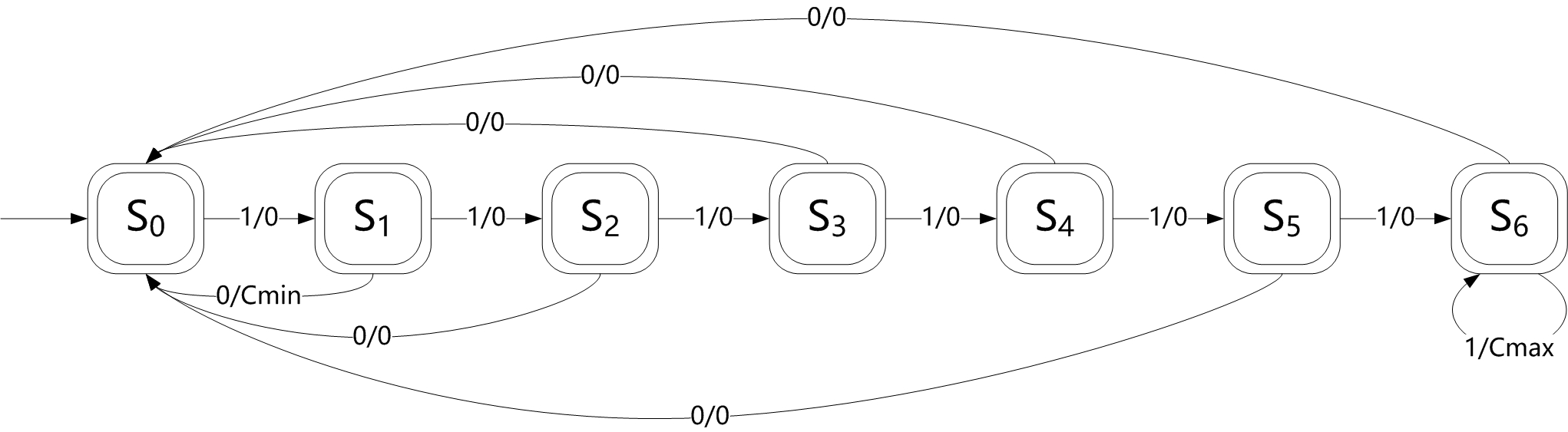}
\caption{Example DFA for the soft constraint 4 with $d_{min} = 2$ and $d_{max} = 6$}
\label{DFA}
\end{figure}

Finally, the labels for the NRPMU are defined as follows.
With each path $P$ from the source node to a node $v_i$, associate a label $L_i = (C_i, R_i^1, R_i^{2u}, R_i^3, S_i^4, S_i^5)$ where $C_i$ denotes the cost and the others are introduced for soft constraints 1-5. $R_i^1, R_i^{2u}, R_i^3$ are resources denoting the number of working days, days worked in unit $u$, and working weekends in the path, respectively. The DFAs for soft constraints 4 and 5 are represented by$S_i^4$ and $S_i^5$, respectively, in their current states. Accordingly, each arc ($v_i,v_j$) in the SPPRC for NRPMU is associated with a cost $c_{ij}$ caused by soft constraints 6-8, the resource consumptions $d_{ij}^1, d_{ij}^{2u}, d_{ij}^3$ for soft constraints 1-3, and inputs $t_{ij}^4, t_{ij}^5$ for the two DFAs. It is straightforward to determine and set these parameters based on whether vertex $v_j$ denotes a working day or a working weekend. Let $s^4/c^4 = \delta^4(S_i^4, t_{ij}^4)$ and $s^5/c^5 = \delta^5(S_i^5, t_{ij}^5)$ be the transition function outputs of corresponding DFAs.
When a path is extended along an arc $(v_i,v_j)$, the label $L_j$ is created by $C_j = C_i + c_{ij} + c^4 + c^5$, $R_j^1 = R_i^1 + d_{ij}^1$, $R_j^{2u} = R_i^{2u} + d_{ij}^{2u}$, $R_j^3 = R_i^3 + d_{ij}^3$,
$S_j^4 = s^4$, and $S_j^5 = s^5$.

\subsection{Dominance rules}
\label{sec_dominance}
Dominance rules are used to reduce the number of labels that need to be considered in the algorithm.  
Without the use of dominance rules, the algorithm would have to enumerate all possible paths, resulting in a brute force algorithm. For a partial path $P = (v_0, ..., v_i)$, we call another partial path $E = (v_j, ...., v_{n+1})$ a feasible extension of path $P$ if the path $(P, E) = (v_0, ..., v_i, v_j, ..., v_{n+1})$ represents a feasible path from the source node to the sink node. Given two partial paths $P_1$ and $P_2$ which both start from the source node, $P_1$ is said to dominate $P_2$ if path $(P_1, E)$ is a feasible path with a lower cost than path $(P_2, E)$ for any feasible extension $E$ of $P_2$. 

Efficient dominance rules have been presented for classical SPPRCs with time or capacity constraints; Nevertheless, these rules cannot be directly applied to solve pricing subproblems with soft constraints specific to NRPs. 
\citet{burke2014new} adapted the dominance rules for classical SPPRCs to treat the maximum and minimum number limits for the same resource, such as working days, separately. That is, for one resource with the maximum (minimum) number limit, a label $A$ dominates another label $B$ if the resource consumption of $A$ is not greater (smaller) than that of $B$.
Obviously, for ranged counter constraits, their method works only when two paths have the same resource consumption.
In contrast, we consider such maximum and minimum limits as a whole and propose tailored dominance rules that can work even if two paths have different resource consumption.

In the remainder of this section, we first present the dominance rules proposed for the NRPMU, followed by necessary explanations and proofs. Let $P_1$ and $P_2$ be two partial paths from $v_0$ to some vertex $v_i$, with labels $L_{P_1} = (C_{i,P_1}, R_{i,P_1}^{1}, R_{i,P_1}^{2u}, R_{i,P_1}^{3}, S_{i,P_1}^{4}, S_{i,P_1}^{5})$ and $L_{P_2} = (C_{i,P_2}, R_{i,P_2}^{1}, R_{i,P_2}^{2u}, R_{i,P_2}^{3}, S_{i,P_2}^{4}, S_{i,P_2}^{5})$, respectively. Given an feasible extension $E = (v_j, ...., v_{n+1})$, we denote as $P_1^{'}=(P_1,E)$ and $P_2^{'} = (P_2,E)$ the individual schedules generated by extending two paths with $E$. The penalty cost caused by the violation of one type of soft constraints when a partial path is extended to a completed individual schedule is denoted as $pn$. According to the definition in section \ref{sec_LableDefinition}, the costs of $P_1^{'}$ and $P_2^{'}$ can be represented by $Cost_{P_1^{'}} = C_{i,P_1} + c_{v_iv_j} + \sum_{(v_k,v_{k+1}) \in E} c_{v_kv_{k+1}} + pn_{P_1^{'}}^1 + \sum_{u \in U} pn_{P_1^{'}}^{2u} + pn_{P_1^{'}}^3 + pn_{P_1^{'}}^4 + pn_{P_1^{'}}^5$ and $Cost_{P_2^{'}} = C_{i,P_2} + c_{v_iv_j} + \sum_{(v_k,v_{k+1}) \in E} c_{v_kv_{k+1}} + pn_{P_2^{'}}^1 + \sum_{u \in U} pn_{P_2^{'}}^{2u} + pn_{P_2^{'}}^3 + pn_{P_2^{'}}^4 + pn_{P_2^{'}}^5$, respectively. Apparently, it is difficult to compare their costs directly because values related to $pn$ are unknown.

To deal with this difficulty, we propose analyzing the difference in costs between $P_1^{'}$ and $P_2^{'}$, which can be expressed as 
$Cost_{P_1^{'}} - Cost_{P_2^{'}} = C_{i,P_1} - C_{i,P_2} +  pn_{P_1^{'}}^1 - pn_{P_2^{'}}^1 + \sum_{u \in U} (pn_{P_1^{'}}^{2u} -  pn_{P_2^{'}}^{2u}) + pn_{P_1^{'}}^3 - pn_{P_2^{'}}^3 + pn_{P_1^{'}}^4 -  pn_{P_2^{'}}^4 + pn_{P_1^{'}}^5 - pn_{P_2^{'}}^5$. If we represent the difference of two '$pn$'s as $pd$, The cost difference of $P_1^{'}$ and $P_2^{'}$ can be rewritten as $C_{i,P_1} - C_{i,P_2} +  pd^1 + \sum_{u \in U} pd^{2u} + pd^3 + pd^4 + pd^5$. In fact, $pd$ is a bounded function whose maximum value $pd_{max}$ and minimum value $pd_{min}$ can be computed. Detailed explanations and proofs are presented in subsequent paragraphs. Assuming that these values are available, we can calculate the maximum cost difference as  $ \max (Cost_{P_1^{'}} - Cost_{P_2^{'}}) = C_{i,P_1} - C_{i,P_2} +  pd^1_{max} + \sum_{u \in U} pd^{2u}_{max} + pd^3_{max} + pd^4_{max} + pd^5_{max}$. If $\max (Cost_{P_1^{'}} - Cost_{P_2^{'}}) < 0$, which means $Cost_{P_1^{'}}$ is always less than $Cost_{P_2^{'}}$, we can determine that $P_1$ dominates $P_2$. Similarly, if $\min (Cost_{P_1^{'}} - Cost_{P_2^{'}}) > 0$, which means  $Cost_{P_1^{'}}$ is always greater than $Cost_{P_2^{'}}$, we can determine that $P_2$ dominates $P_1$.

The following paragraphs aim to clarify whether $pd$ is a bounded function and how to compute its maximum and minimum values if it is.
\color{black}
We take soft constraints 1, which limit the maximum ($d_{max}$) and minimum ($d_{min}$) number of working days, as an example to illustrate the process to analyze $pd$. Consider two partial paths $P_1$ and $P_2$ and two complete paths $P_1^{'}$ and $P_2^{'}$, which are generated by extending $P_1$ and $P_2$ with $E$, respectively. 
With loss of generality, assume that $R_{i,P_1}^{1} < R_{i,P_2}^{1}$. In Figure \ref{six}, the line segment AD represents the scheduling period and red line segments denote time intervals between $R_{i,P_1}^{1}$ and $R_{i,P_2}^{1}$. Depending on their values, there are six possible cases where the endpoints of a red line segment are located at different parts of AD.
%
Let $d_e$ be the number of working days in the extension $E$ and $d_c$ be the day index of the last nodes in paths $P_1$ and $P_2$. It is easy to understand that the value of $d_e$ ranges from 0 to $|D|-d_c$.
Let $R_{n+1,P_1^{'}}^{1}$ and $R_{n+1,P_2^{'}}^{1}$ be the final numbers of working days in paths $P_{1}^{'}$ and $P_{2}^{'}$ respectively.
As the value of $d_e$ changes from 0 to $|D|-d_c$, the value of $R_{n+1,P_1^{'}}^{1}$ ($R_{n+1,P_2^{'}}^{1}$) changes from $R_{i,P_1}^{1}$ ($R_{i,P_2}^{1}$) to $R_{i,P_1}^{1} + |D| - d_c$ ($R_{i,P_2}^{1} + |D| -d_c$). Moreover, once the value of $d_e$ is determined, the $pd^1$ can be computed. In other words, $pd^1$ is a function of $d_e$ (denoted $pd = f(d_e)$).

\begin{figure}[htbp]
\centering
\includegraphics[width = .8\textwidth]{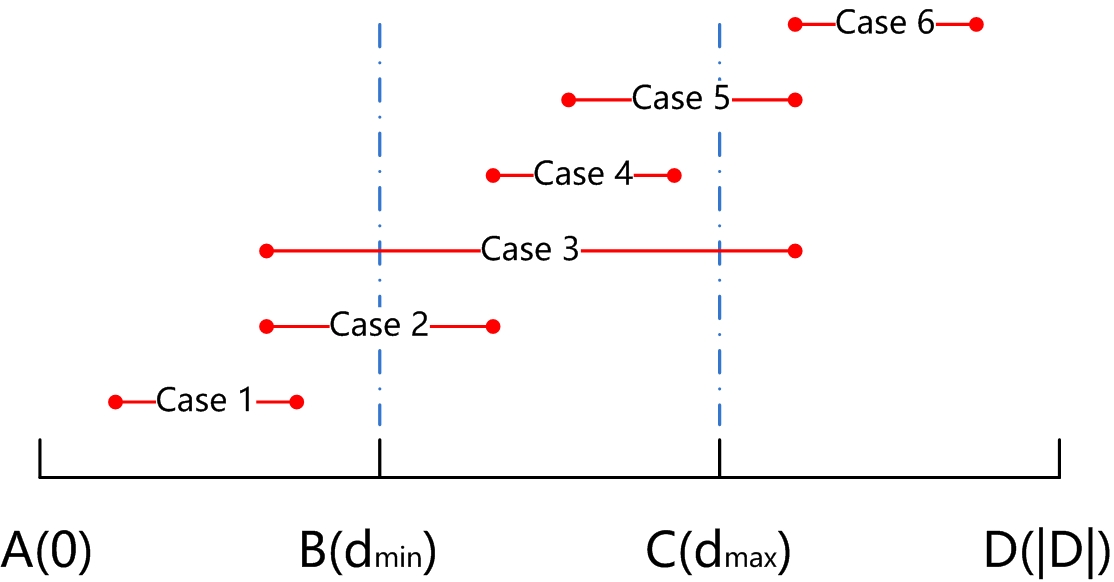}
\caption{Six cases depending on  the values of $R_{i,P_1}^{1}$ and $R_{i,P_2}^{1}$}
\label{six}
\end{figure}

%
Using case 1 in Figure \ref{six} as an example, we illustrate the effect of the variation of $d_e$ in Figure \ref{move}. The red line segment, which denotes the time interval between $R_{n+1,P_1^{'}}^{1}$ and $R_{n+1,P_2^{'}}^{1}$, shifts forward based on the magnitude of $d_e$. As shown in Figure \ref{move}, the length of the time interval remains unchanged. This variation process can be divided into five stages based on the method to compute $pd^1$. Accordingly, Figure \ref{axis} shows the graph of the function $pd^1 = f(d_e)$, where $d_e$ ranges from 0 to $|D| - d_c$. It is worth noting that the function is a decreasing piecewise function, which consists of five pieces.

In fact, regardless of initial values of $R_{i,P_1}^{1}$ and $R_{i,P_2}^{1}$, $pd^1$ is always a decreasing function of $d_e$ if $R_{i,P_1}^{1} < R_{i,P_2}^{1}$. The proof is given in Claim \ref{proof}. Therefore, when $d_e$ takes the values of 0 and $|D|-d_c$, $pd^1$ reaches its maximum and minimum values, respectively. Conversely, if $R_{i,P_1}^{1} > R_{i,P_2}^{1}$, it is not difficult to understand that $pd^1$ is an increasing function of  $d_e$. As a result, $pd^1$ reaches its maximum at $|D|-d_c$ and minimum at 0.  $pd^1$ is always equal to 0 When $R_{i,P_1}^{1} = R_{i,P_2}^{1}$. 

For the sake of clarity, in Claim \ref{proof} and its proof, we use $pd$, $R_1$, $R_2$, $R_1^{'}$ and $R_2^{'}$  as shorthand notations for  $pd^1$, $R_{i,P_1}^{1}$, $R_{i,P_2}^{1}$, $R_{n+1,P_1^{'}}^{1}$ and $R_{n+1,P_2^{'}}^{1}$, respectively.

\begin{figure}[htbp]
\centering
\includegraphics[width = .8\textwidth]{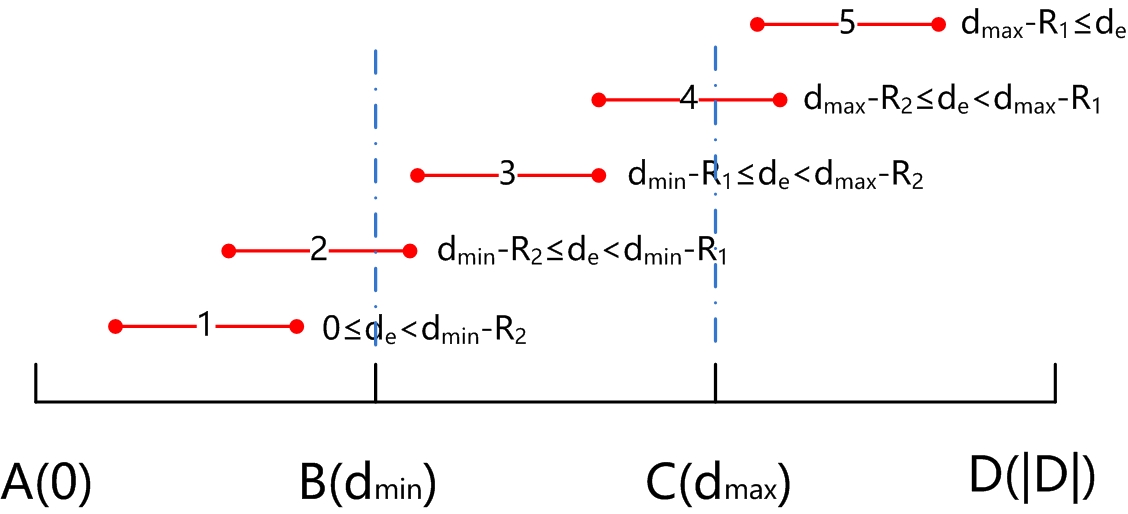}
\caption{A diagram showing how  the line segment for case 1 shifts forward with the variation of $d_e$ ($R_1$ and $R_2$ are shorthand notations for $R_1^{'}$ and $R_2^{'}$, respectively)}
\label{move}
\end{figure}

\begin{figure}[htbp]
\centering
\includegraphics[width = .8\textwidth]{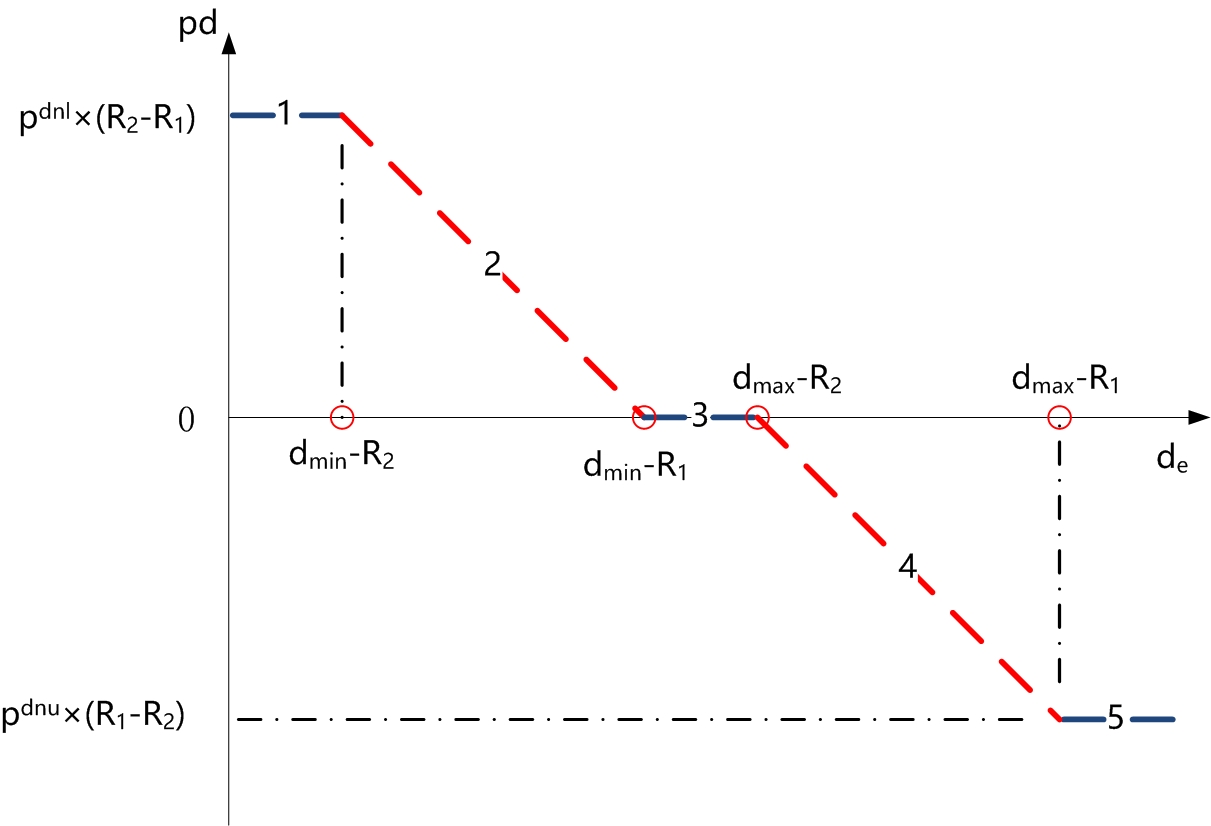}
\caption{A diagram showing how the value of $pd$ for case 1 changes with the
variation of $d_e$ ($R_1$ and $R_2$ are shorthand notations for $R_1^{'}$ and $R_2^{'}$, respectively)}
\label{axis}
\end{figure}

\begin{theorem}
\label{proof}
Given two paths $P_1$ and $P_2$ with associated resource consumptions $R_1$ and $R_2$, the $pd$ is a decreasing function of $d_e$ if $R_1 < R_2$.
\end{theorem}

\begin{proof}
It is easy to see that each stage in Figure \ref{move} corresponds to a case in Figure
\ref{six}. The process to shift forward a red line segment as $d_e$ increases can be considered as the process to change a line segment between different cases. Table \ref{TabStates} presents the formula to compute $pd$ and next possible cases for each case. Note that two formulas are given for each case and the second formula is obtained by replacing $R_1^{'}$ ($R_2^{'}$) in the first one with $R_1 + d_e$ ($R_2+d_e$).

As discussed above, $pd$ is a function of $d_e$ ($pd = f(d_e)$). Once the value of $d_e$ is determined, we are able to compute $f(d_e)$ by one formula corresponding to the case of the values of $R_1^{'}$ and $R_2^{'}$. Based on the feasible range of $d_e$, the function $f(d_e)$ can consist of one or several sub-functions, each of which corresponds to one formula in Table \ref{TabStates}. It is obvious that the result value of each formula remains unchanged or decreases as $d_e$ increases, that is, every sub-function is a decreasing function. In addition, the function $f(d_e)$ is continuous at boundary points where case change occurs if several sub-functions are involved. For instance, at the boundary point ($d_e = d_{min} - R_1$) between cases 3 and 5, the results of the formulas for cases 3 and 5 are both $(d_{max}-R_2-d_{min}+R_1) \times p^{dnu}$. Finally, we conclude that since the function $pd = f(d_e)$ is a continuous function made up of one or several decreasing sub-functions, it is a decreasing function.
The proof is completed. 
\end{proof}

\begin{table}[]
\centering
\addtolength{\leftskip} {-2cm}
\addtolength{\rightskip}{-2cm}
\begin{tabular}{ccc}
\hline
Current case  & $pd$  & Next possible cases   \\ \hline
\multirow{2}{*}{Case 1} & $((d_{min} - R_1^{'}) - (d_{min} - R_2^{'})) \times p^{dnl}$   & \multirow{2}{*}{2}    \\   & \multicolumn{1}{c}{$ = (R_2 - R_1) \times p^{dnl}$}              &   \\
\multirow{2}{*}{Case 2} & $(d_{min}-R_1^{'}) \times p^{dnl}$                                & \multirow{2}{*}{3 or 4}      \\      & \multicolumn{1}{c}{$= (d_{min}-R_1-d_e) \times p^{dnl}$ }                                         &                                                       \\
\multirow{2}{*}{Case 3} &  $ (d_{min}-R_1^{'}) \times p^{dnl} - (R_2^{'} - d_{max}) \times p^{dnu}$ & \multirow{2}{*}{5}                                    \\
                        & \multicolumn{1}{c}{$=(d_{min}-R_1) \times p^{dnl} + (d_{max}-R_2) \times p^{dnu}   - (p^{dnl} + p^{dnu}) \times d_e$}    &    \\
Case 4                  & 0                                 & 5                  \\
\multirow{2}{*}{Case 5} & $(d_{max}-R_2^{'}) \times p^{dnu}$                                & \multirow{2}{*}{6}                                    \\
                        & \multicolumn{1}{c}{ $=(d_{max}-R_2-d_e) \times p^{dnu}$ }   & \\
\multirow{2}{*}{Case 6} & \multicolumn{1}{c}{$((R_1^{'}-d_{max})-(R_2^{'}-d_{max})) \times p^{dnu}$}                                         & \multicolumn{1}{c}{\multirow{2}{*}{\textbackslash{}}} \\
                        & $=(R_1 - R_2) \times p^{dnu}$ & \multicolumn{1}{l}{}        \\ \hline
\end{tabular}
\caption{Changes between different cases}
\label{TabStates}
\end{table}

The methods provided above can be applied to soft constraints 1, 2, 4, and 5 for the NRPMU. Additionally, these methods remain effective for other soft NRP constraints with maximum and minimum number limits. Identical ideas can be used to analyze $pd^3$ for soft constraint 3 that limits only the maximum number of working weekends. Its maximum and minimum values are not challenging to obtain.

\subsection{How our method differs from Omer and Legrain’s}
\label{subsec_differ}
The label definition and dominance rules we proposed have been described in section \ref{sec_LableDefinition} and \ref{sec_dominance}. In the following, we outline the differences between our method and the one proposed by \citet{omer2023dedicated} in terms of these two aspects. Notice that they present an outstanding and systematic work, which includes topics not covered in our manuscript. For instance, we do not discuss how to handle hard versions of related constraints. Here, we only compare the methods involved in both works.

{\bf The label definitions}: the first difficulty we encounter is how to define and extend the label, which is related to how to define resource consumption and cost associated with arcs of the graph constructed for the shortest path problem with resource constraints. The counter constraint can be easily handled by treating them as one common resource like the time in vehicle routing problems. In contrast, the series constraint is hard to deal with. To address the series constraint, \citet{omer2023dedicated} develop a set of tailored rules to define and extend labels. In contrast to their method, we introduce the deterministic finite automaton (DFA) to deal with labels.

Although these two methods are almost the same in essence for most constraints, the DFA-based method is more efficient when approaching forbidden pattern constraints. To handle a forbidden pattern constraint when extending a label to a new vertex $v_i$, \citet{omer2023dedicated} present “Algorithm 2” to find any subpatterns and determine the resource value at $v_i$. The time complexity of “Algorithm 2” is $O(n)$. In contrast, the DFA-based method treats this process as a state transition, which can be completed in $O(1)$ time according to corresponding state transition tables. Using the DFA-based method can be more valuable when dealing with problems that have a number of forbidden patterns. Additionally, multiple forbidden patterns may have overlapping DFA states. There are mature algorithms available to decrease the number of states \citep{pesant2004regular}, thereby improving algorithm efficiency. 

{\bf Dominance rules}: the second difficulty we encounter is how to define dominance rules to reduce the number of labels extended in the DP algorithm. \citet{omer2023dedicated} present dominance rules for both hard and soft constraints. However, we only compare our dominance rules with theirs with respect to soft constraints. To clarify the differences, we directly adopt the symbols they defined. 

Let us recall some definitions first. Consider two partial paths  $P$ and $Q$ from $o$ to some vertex $v \in V_i$  (vertices of the graph for nurse $i$). They denote as $\gamma^{0}(P)$  the cost of $P$ and  $\gamma^{r}(P)$ the consumption of resource $r\in R$  along $P$ . Let $\overline{Q}$ be a completion of $P$  and $Q$ , which means $[P,{\overline{Q}}]$ and $[Q,{\overline{Q}}]$ are paths from the origin vertex $o$ to the sink vertex $t$. For any soft resource  $r$, Let $G^{r}(P,{\overline{Q}})$  and $G^{r}(Q,{\overline{Q}})$  be the penalty costs resulting from violations of resource $r$ along $\overline{Q}$  in paths $[P,{\overline{Q}}]$  and $[Q,{\overline{Q}}]$, respectively. Here, we define $c_{{\overline{Q}}}$  as the sum of arc costs when extending $P$ or $Q$ from $v$  to the sink vertex  $t$ along $\overline{Q}$ . Then, it is easy to see that the final costs of feasible paths   $[P,{\overline{Q}}]$ and $[Q,{\overline{Q}}]$ are $\gamma^{0}([P,{\overline{Q}}])=\gamma^{0}(P){+}c_{ \overline{Q}}+\sum_{r\in{\cal R}}G^{r}(P,{\overline{Q}})$ and $\gamma^{0}([Q,{\overline{Q}}])=\gamma^{0}(Q){+}c_{ \overline{Q}}+\sum_{r\in{\cal R}}G^{r}(Q,{\overline{Q}})$. Obviously, we cannot directly compare their final costs since   $G^{r}(P,{\overline{Q}})$  and $G^{r}(Q,{\overline{Q}})$ are unknown. 

The key idea to handle this difficulty is to analyze the difference between the final costs of these two paths:$\gamma^{0}([P,{\overline{Q}}])-\gamma^{0}([Q,{\overline{Q}}]) = \gamma^{0}(P)+c_{\overline{Q}}+\sum_{r \in R}G^{r}(P,{\overline{Q}})-(\gamma^{0}(Q)+c_{\overline{Q}}+\sum_{r\in R}G^{r}(Q,{\overline{Q}}))$, which is equal to $\gamma^{0}(P) - \gamma^{0}(Q) + \sum_{r \in R}(G^{r}(P,{\overline{Q}}) - G^{r}(Q,{\overline{Q}}))$. \citet{omer2023dedicated} introduce $\Delta^{r}(P,Q,{\overline{{Q}}})$ to be the difference between $G^{r}(P,{\overline{Q}})$  and $G^{r}(Q,{\overline{Q}})$. Let $x$  represent the resource consumption of $r$  along $\overline{Q}$. $\Delta^{r}(P,Q,{\overline{{Q}}})$  is indeed a function of $x$  (denoted as $\Delta^{r}(P,Q,{\overline{{Q}}}) = f(x)$), whose range of variation is available. Apparently, the minimum resource consumption is 0. Let $\overline{D}$  denote the maximum possible resource consumption. If we take a constraint on the number of working days as an example, $\overline{D}$  represents the number of days along $\overline{Q}$. This means that the maximum resource consumption is achieved by working every day along $\overline{Q}$. 

For the moment, all the concepts mentioned above are also discussed in our manuscript although we used different symbols and examples. With the notation defined above, we start to explain the differences between our dominance rules and theirs in detail. We both realize that as $x$ varies from 0 to $\overline{D}$, $f(x)$ is a bounded function. However, they directly present the formulas (indexes with (9)) to compute its upper bound. In contrast, we first present a proof in “Claim 1” to prove that $f(x)$ is a deceasing (increasing) function of $x$ if $\gamma^{r}(P) < \gamma^{r}(Q)$ ($\gamma^{r}(P) > \gamma^{r}(Q)$, note that $f(x)$ is always equal to 0 when $\gamma^{r}(P) = \gamma^{r}(Q)$, so we do not discuss it in our proof). As a result, we know that $f(x)$ reaches its upper bound at $x=0$ ($\overline{D}$) when $\gamma^{r}(P) < \gamma^{r}(Q)$ ($\gamma^{r}(P) > \gamma^{r}(Q)$).
Since there are different situations where initial resource consumptions of two partial paths and their variations are different, it is not so easy to know these conclusions. In our opinion, the proof is the theoretical foundation of the dominance rules we developed.

The proof helps us to present more efficient dominance rules. Their method can only compare two paths with different costs in one direction: it can determine whether the path with lower cost dominates the path with higher cost, but not vice versa. In contrast, our method can compare two paths in both directions and determine whether either path dominates the other regardless of their costs. Let us explain it in detail. They denote $\overline{\Delta}^{r}(P,Q,\overline{Q})$ as the upper bound. According to their formulas “(9)”, $\overline{\Delta}^{r}(P,Q,\overline{Q})$ is always greater than or equal to 0; Hence, the “Property 1” in their manuscript implicitly assumes that the initial cost of $P$ is less than that of $Q$, that is $\gamma^{0}(P)<\gamma^{0}(Q)$. Otherwise, it is impossible to achieve $\gamma^{0}(P) + \sum_{r \in R} \overline{\Delta}^{r}(P,Q,\overline{Q}) <\gamma^{0}(Q)$ to determine that $P$ dominates $Q$. In other words, their method cannot be applied to determine whether the path with higher cost dominates the one with lower cost.
Based on our proof, we know that $f(x)$ reaches its lower bound at $x=\overline{D}$ ($0$) when $\gamma^{r}(P) < \gamma^{r}(Q)$ ($\gamma^{r}(P) > \gamma^{r}(Q)$). In addition to the upper bound, we compute the lower bound $\underline{\Delta}^{r}(P,Q,\overline{Q})$ for each resource $r$. As a consequence, we can determine that $Q$ dominates $P$ if $\gamma^{0}(P) + \sum_{r \in R} \underline{\Delta}^{r}(P,Q,\overline{Q}) > \gamma^{0}(Q)$, which means $\gamma^{0}([P,\overline{Q}]) - \gamma^{0}([Q,\overline{Q}])$ is always greater than 0. That is, even if $Q$ is the path with higher cost, it is possible to identify $Q$ dominates $P$ provided that $\sum_{r \in R} \underline{\Delta}^{r} (P,Q,\overline{Q}) > \gamma^{0}(Q) - \gamma^{0}(P) > 0$.
This is the situation \citet{omer2023dedicated} do not consider.

According to the analysis presented above, it is easy to understand that the first step of applying their dominance rules is to identify the path with lower cost as $P$ and the other one as $Q$. It is notable that our dominance rules do not distinguish $P$ and $Q$ based on their costs. The reason is that the upper bound is equal to the lower bound when we permute $P$ and $Q$.

Finally, our method computes the upper bound differently when $\gamma^{r}(P) > \gamma^{r}(Q)$. According to our proof and method, the upper bound is possible less than 0 in theory. Consider the constraint limiting the maximum 6 and minimum 3 number of working days and two partial paths $P$ and $Q$ with 3 and 1 working days, respectively. Assuming that $\overline{D}$ is 1, the upper bound is achieved at 1, that is $0-c_{r}^{L}(3-2)$. This may help when multiple resources are involved. 

\subsection{Applying dominance rules across multiple nodes}
\label{sec_multiplenodes}
classical DP algorithms for SPPRCs apply dominating rules within one node to reduce surplus labels. The dominance rules provided above are also based on two partial paths $P_1$ and $P_2$ that arrive at the same node $v_i$. However, as illustrated in Figure \ref{graph}, since one node is created for each unit and shift in the directed graph for the NRPMU, the number of nodes on the same day is greater than that for the NRP with single unit. Consequently, the algorithm efficiency gradually decreases as the number of units increases.

To achieve a better performance, an accelerating strategy specific to the NRPMU is provided. Since different units have the same shift types in our NRPMU instances, nodes representing the same shift type in different units on the same day have the same set of arcs to nodes on the next day. Therefore, it is easy to understand that one feasible extension $E$ of path $P_1$ whose last node corresponds to one shift in one unit is also feasible for path $P_2$ with the last node denoting one identical shift type in another unit. As a result, it is possible to apply the dominating rules to compare labels across multiple nodes corresponding to the same shift type in different units, which contributes to reducing more labels and improving algorithm performance. 
\color{black}

\subsection{Description of the DP algorithm and $B\&P$}
Since the graph constructed for the pricing subproblem is directed and acyclic, it is eay to determine the sequence in which the nodes are processed. For the sake of clarity, the algorithm that applies dominance rules within one single node is presented in Algorithm \ref{Alg}, based on which it is not hard to implement the accelerating strategy introduced in section \ref{sec_multiplenodes}. Related notations are given in the following.

\begin{itemize}
    \item $\Phi_i$: set of labels on node $v_i$.
    \item $Suc(v_i)$: list of  successors of node $v_i$.
    \item $E$: set of all nodes sorted by the day index.
    \item $REF(\phi_i,v_j)$: Function that returns a label generated by extending labels $\phi_i \in \Phi_i$ to node $v_j$ if the extension is feasible. The resource consumptions and DFA states in the label are updated.
    \item $DMN(\Phi_i)$ Procedure that removes labels that are dominated by others in the list $\Phi_i$.
\end{itemize}

\begin{algorithm}
    \caption{the DP algorithm for the pricing subproblem}
    \label{Alg}
    \begin{algorithmic}[1]
        \STATE Initialization:
        \STATE $\Phi_0 \gets {(0,...,0)}$
        \FOR{$v_i \in V \backslash \{v_0\}$ }
            \STATE $\Phi_i \gets \emptyset$
        \ENDFOR
        \STATE Extending labels
        \FOR{$v_i \in V $}
          \STATE $\Phi_i \gets DMN(\Phi_i)$
          \FOR{$v_j \in Suc(v_i)$}
            \FOR{$\phi_i \in \Phi_i$}
                \STATE $\Phi_j \gets \Phi_j \cup REF(\phi_i,v_j)$
            \ENDFOR
          \ENDFOR
        \ENDFOR
    \end{algorithmic}
\end{algorithm}

The $B\&P$ algorithm implementation follows the general $B\&P$ framework that has been described by a large number of papers \citep{feillet2010tutorial, vaclavik2018accelerating}. Based on the work by \citet{burke2014new}, the branching strategy used is to branch on individual nurse-day-unit-shift assignments $\theta_{ndus}$. When the MP is solved by the CG, we compute $\theta_{ndus} = \sum_{l \in L_n} x_{nl} \cdot a_{nldus}$ and select the assignment with the value closest to 1. Then, two child nodes are generated by adding constraints that restrict whether to select this assignment or not. 
To handle these constraints in the pricing subproblem, one can either remove corresponding arcs in the graph or force their selection.
A simple deterministic heuristic is applied to provide initial solutions to the root node. In addition, several strategies are adopted to improve the algorithm performance. Given that the pricing procedure consists of a set of subproblems and iterates many times, a heuristic strategy is developed to reduce the number of subproblems to be solved in one iteration. That is, if the reduced cost of a subproblem is not negative in one iteration, it will be ignored in the following iterations. All subproblems will be checked again in the final iteration so that the optimal solution is found. Also, Multithreading is employed to solve several subproblems simultaneously and speed up the pricing process.

\section{Computational experiments}
\label{sec_experiments}
To assess the efficiency of the proposed algorithm to solve NRPMU, two experiments have been conducted. First, we test the performance of the DP algorithm presented in section \ref{sec_solution} to solve pricing subproblems and compare it with the DP algorithm proposed by \cite{burke2014new}, \cite{omer2023dedicated}, and a commercial integer programming solver (CIPSolver). Second, we analyze the benefit of considering multiple units and assess the efficiency of the $B\&P$ algorithm to solve NRPMU instances.
Since there are no available benchmark instances, a total of 30 NRPMU instances are created based on the problem definition and real world requirements. Their sizes vary from small (10 nurses, 2 weeks, 2 units) to large (50 nurses, 4 weeks, 4 units), and Table \ref{TabBenchmark} provides more information regarding their dimensions. Note that we assume that there are three identical  shifts (early, late, and night) in each unit for all instances. These instances are given in an XML format (available online).

\begin{table}[]
\centering
\begin{tabular}{cccccccc}
\hline
Instances & Nurses & Weeks & Units & Instances & Nurses & Weeks & Units \\ \hline
1         & 10     & 2     & 2     & 16        & 10     & 4     & 2     \\
2         & 20     & 2     & 2     & 17        & 20     & 4     & 2     \\
3         & 30     & 2     & 2     & 18        & 30     & 4     & 2     \\
4         & 40     & 2     & 2     & 19        & 40     & 4     & 2     \\
5         & 50     & 2     & 2     & 20        & 50     & 4     & 2     \\
6         & 10     & 2     & 3     & 21        & 10     & 4     & 3     \\
7         & 20     & 2     & 3     & 22        & 20     & 4     & 3     \\
8         & 30     & 2     & 3     & 23        & 30     & 4     & 3     \\
9         & 40     & 2     & 3     & 24        & 40     & 4     & 3     \\
10        & 50     & 2     & 3     & 25        & 50     & 4     & 3     \\
11        & 10     & 2     & 4     & 26        & 10     & 4     & 4     \\
12        & 20     & 2     & 4     & 27        & 20     & 4     & 4     \\
13        & 30     & 2     & 4     & 28        & 30     & 4     & 4     \\
14        & 40     & 2     & 4     & 29        & 40     & 4     & 4     \\
15        & 50     & 2     & 4     & 30        & 50     & 4     & 4     \\ \hline
\end{tabular}
\caption{Benchmark instances}
\label{TabBenchmark}
\end{table}

Computational experiments were carried out on a Windows 10 PC with an Intel Core i7 1.8-GHz processor and 16 GB of RAM.
Algorithms were implemented  in Java using JDK 15.0.2. The IBM ILOG CPLEX 22.1.0 is used as the CIPSolver to solve related linear programming or integer programming models. 
The maximum computational time was set to 15 seconds in the first experiment and to 3600 seconds in the third one.

\subsection{Performance of different exact methods for the pricing subproblem }

The first experiment aims to test the performance of the DP algorithm proposed in this paper to solve pricing subproblems. Given the impact of dual variables on problem difficulty, all dual variables $\lambda_{dus}$ are set to zero.
As discussed above, there have been several methods to solve it. For example, using a CIPSolver is able to solve it based on the model given in section 2.2. Our DP algorithm can solve pricing subproblems to optimality, so we compare it with three exact methods, namely, using a CIPSolver and the DP algorithms proposed by \citet{burke2014new} and \citet{omer2023dedicated}. 
As their DP algorithms were initially proposed to deal with general NRP instances, some modifications have to be made to solve NRPMU ones. 

Given that the method presented by \citet{burke2014new} is characterized by their dominating rules that address the maximum and minimum number limit of the same resource separately, we adopt the label definition in this paper when implementing their DP algorithm. In other words, we implementing their DP algorithm that differs from ours only in dominating rules.
The DP algorithm described in algorithm \ref{Alg} is called DPP, while the one with dominating rules developed by \citet{burke2014new} is called DPB.
As stated in the section \ref{subsec_differ}, the main difference between the method proposed by \citet{omer2023dedicated} and ours is that their
method can only compare two paths with different costs in one direction using the upper bound. Hence, we adapt our DPP by limiting the dominance rules to only consider the upper bound. The adapted DPP is referred to as DPU.

To achieve a better performance, the accelerating strategy to apply dominance rules across multiple nodes is proposed in section \ref{sec_multiplenodes}. We refer to the DP algorithm implementing this strategy as the improved version of DPP (DPPI).

\begin{table}[]
\centering
\begin{tabular}{ccccc}
\hline
Sub Instances & Instance ID & Nurse ID & Weeks & Units \\ \hline
Sub1          & 2           & 1        & 2     & 2     \\
Sub2          & 3           & 2        & 2     & 2     \\
Sub3          & 3           & 6        & 2     & 2     \\
Sub4          & 5           & 12       & 2     & 2     \\
Sub5          & 5           & 43       & 2     & 2     \\
Sub6          & 7           & 2        & 2     & 3     \\
Sub7          & 8           & 19       & 2     & 3     \\
Sub8          & 9           & 33       & 2     & 3     \\
Sub9          & 10          & 1        & 2     & 3     \\
Sub10         & 10          & 2        & 2     & 3     \\
Sub11         & 12          & 2        & 2     & 4     \\
Sub12         & 13          & 11       & 2     & 4     \\
Sub13         & 14          & 25       & 2     & 4     \\
Sub14         & 15          & 26       & 2     & 4     \\
Sub15         & 15          & 48       & 2     & 4     \\
Sub16         & 17          & 6        & 4     & 2     \\
Sub17         & 17          & 20       & 4     & 2     \\
Sub18         & 18          & 16       & 4     & 2     \\
Sub19         & 18          & 18       & 4     & 2     \\
Sub20         & 19          & 8        & 4     & 2     \\
Sub21         & 21          & 3        & 4     & 3     \\
Sub22         & 21          & 9        & 4     & 3     \\
Sub23         & 22          & 3        & 4     & 3     \\
Sub24         & 24          & 10       & 4     & 3     \\
Sub25         & 25          & 4        & 4     & 3     \\
Sub26         & 26          & 5        & 4     & 4     \\
Sub27         & 27          & 2        & 4     & 4     \\
Sub28         & 28          & 24       & 4     & 4     \\
Sub29         & 29          & 40       & 4     & 4     \\
Sub30         & 30          & 49       & 4     & 4     \\ \hline
\end{tabular}
\caption{Instances for the pricing subproblem}
\label{TabInstancesSub}
\end{table}

Clearly, the instances introduced in Table 3 cannot be directly used since the goal of the pricing subproblem is to create an individual schedule for only one nurse. Table \ref{TabInstancesSub} lists the instances to test pricing subproblem solutions, each of which corresponds to one nurse selected from one instance in Table 3. To identify these instances across two experiments, they are named with the prefix “Sub”. The Columns Instance ID and Nurse ID indicate how the subproblem instances and instances in Table 3 are connected.

\begin{table}[]
\centering
\addtolength{\leftskip} {-2cm}
\addtolength{\rightskip}{-2cm}
\begin{tabular}{ccccccccccc}
\hline
\multirow{2}{*}{Sub Instances} & \multicolumn{5}{c}{Time (millisecond)} &  & \multicolumn{4}{c}{The number of labels} \\ \cline{2-11} 
                               & Cplex  & DPB    & DPU   & DPP  & DPPI  &  & DPB       &    DPU      & DPP      & DPPI   \\ \hline
Sub1                           & 34     & 299    & 11    & 4    & 3     &  & 6552      & 810      & 412      & 232    \\
Sub2                           & 81     & 457    & 11    & 3    & 3     &  & 11331     & 732      & 358      & 202    \\
Sub3                           & 54     & 234    & 2     & 3    & 2     &  & 11112     & 544      & 336      & 189    \\
Sub4                           & 25     & 91     & 1    & 3    & 2     &  & 6942      & 584      & 395      & 218    \\
Sub5                           & 32     & 125    & 2     & 2    & 1     &  & 7309      & 803      & 426      & 243    \\
Sub6                           & 60     & 5170   & 154   & 31   & 33    &  & 43971     & 4690     & 2297     & 1109   \\
Sub7                           & 31     &        & 2     & 3    & 1     &  &           & 839      & 508      & 200    \\
Sub8                           & 25     & 4027   & 4     & 4    & 2     &  & 42884     & 1170    & 592      & 232    \\
Sub9                           & 79     &        & 18    & 7    & 4     &  &           & 1071     & 514      & 202    \\
Sub10                          & 28     & 3391   & 20    & 8    & 3     &  & 40737     & 828      & 548      & 212    \\
Sub11                          & 52     &        & 18    & 9    & 3     &  &           & 1530     & 772      & 232    \\
Sub12                          & 99     &        & 15    & 7    & 5     &  &           & 1410    & 670      & 202    \\
Sub13                          & 50     &        & 11    & 4    & 3     &  &           & 1408    & 668      & 200    \\
Sub14                          & 38     &        & 14    & 4    & 3     &  &           & 1251     & 636      & 195    \\
Sub15                          & 38     &        & 12    & 5    & 4     &  &           & 1530     & 772      & 232    \\
Sub16                          & 106    & 9101   & 16    & 10   & 4     &  & 105494    & 3719     & 1444     & 820    \\
Sub17                          & 105    & 10140  & 11    & 5    & 3     &  & 103461    & 2912     & 1405     & 802    \\
Sub18                          & 117    & 3888   & 51    & 8    & 10    &  & 68688     & 4479     & 1694     & 956    \\
Sub19                          & 130    & 2818   & 40    & 6    & 9     &  & 57861     & 5110     & 1722     & 960    \\
Sub20                          & 135    & 11846  & 14    & 8    & 4     &  & 104914    & 3625     & 1462     & 829    \\
Sub21                          & 153    &        & 32    & 20   & 16    &  &           & 4831     & 2056     & 816    \\
Sub22                          & 160    &        & 2783  & 819  & 534   &  &           & 51072   & 19577    & 8558   \\
Sub23                          & 176    &        & 63    & 20   & 22    &  &           & 5411     & 2068     & 820    \\
Sub24                          & 132    &        & 66    & 15   & 8     &  &           & 7423     & 2484     & 960    \\
Sub25                          & 137    &        & 60    & 18   & 10    &  &           & 5579     & 1937     & 779    \\
Sub26                          & 125    &        & 165   & 36   & 17    &  &           & 6462     & 2740     & 832    \\
Sub27                          & 180    &        & 152   & 33   & 34    &  &           & 7103     & 2692     & 820    \\
Sub28                          & 127    &        & 151   & 30   & 18    &  &           & 9736     & 3246     & 960    \\
Sub29                          & 115    &        & 82    & 34   & 10    &  &           & 6729    & 2694     & 822    \\
Sub30                          & 104    &        & 84   & 31   & 14    &  &           & 6273     & 2607     & 798    \\ \hline
\end{tabular}
\caption{Results for the pricing subproblem instances}
\label{TabSubResults}
\end{table}

These instances are solved by five methods, namely, CIPSolver (Cplex), DPB, DPU, DPP, and DPPI. The results are presented in Table \ref{TabSubResults}. The comparison is based on the computational time (in milliseconds) and the number of labels to be extended. The second criterion is only applicable for the DP based algorithms, and fewer labels mean that the dominating rules perform better. The cells are empty when the computational time limit is exceeded. As we can see, DPP succeeds in solving almost all instances within 40 milliseconds, and the computation time and the label count increase with the number of units or weeks. In contrast, DPB can solve only instances with 2 weeks and no more than 3 units or with 4 weeks and 2 units within the given time limit. Its performance drops sharply as the instance size grows. Although Cplex solves all instances, it is time-consuming. 
With respect to the number of labels, compared to DPU, DPP is capable of reducing more labels, averaging $53.8\%$. As expected, DPPI extends fewer labels than DPP in all instances.
On average, the number of extended labels decreases by about $57.6\%$. These results confirm that the ideas presented above to develop DPP and DPPI are effective and efficient.

\subsection{Comparision with commercial solvers}

\begin{table}[]
\centering
\begin{tabular}{ccccccccc}
\hline
\multirow{2}{*}{Instance} & \multirow{2}{*}{Single} & \multicolumn{3}{c}{Cplex}   &  & \multicolumn{3}{c}{ $B\&P$} \\ \cline{3-5} \cline{7-9} 
                          &                         & LB    & UB           & Time &  & LB        & UB              & Time   \\ \cline{1-9} 
1                         & 130             & 104   & \textbf{122} & 17   &  & 122       & \textbf{122}    & 5      \\
2                         & 177                & 106   & \textbf{133} & 84   &  & 133       & \textbf{133}    & 10     \\
3                         & 198              & 108   & \textbf{131} & 1196 &  & 130.5     & \textbf{131}    & 21     \\
4                         & 388                 & 282   & \textbf{353} & 129  &  & 353       & \textbf{353}    & 15     \\
5                         & 522                  & 349   & \textbf{430} & 259  &  & 430       & \textbf{430}    & 20     \\
6                         & 153                  & 116   & \textbf{135} & 4    &  & 135       & \textbf{135}    & 12     \\
7                         & 189.5               & 130   & \textbf{158} & 28   &  & 158       & \textbf{158}    & 15     \\
8                         & 266                 & 207   & \textbf{252} & 167  &  & 252       & \textbf{252}    & 14     \\
9                         & 384                  & 233   & \textbf{273} & 251  &  & 272.67    & \textbf{273}    & 29     \\
10                        & 476                  & 313   & \textbf{363} & 448  &  & 363       & \textbf{363}    & 58     \\
11                        & 168                  & 111   & \textbf{129} & 8    &  & 129       & \textbf{129}    & 33     \\
12                        & 144                  & 76    & \textbf{86}  & 657  &  & 86        & \textbf{86}     & 39     \\
13                        & 224                 & 136   & 154          & 3600 &  & 153.25    & \textbf{154}    & 109    \\
14                        & 326                  & 242   & \textbf{298} & 264  &  & 298       & \textbf{298}    & 49     \\
15                        & 474                & 368   & \textbf{466} & 265  &  & 466       & \textbf{466}    & 30     \\
16                        & 252                  & 196   & \textbf{228} & 39   &  & 228       & \textbf{228}    & 131    \\
17                        & 334                  & 224   & 266          & 3600 &  & 265.33    & \textbf{266}    & 427    \\
18                        & 456                  & 296   & 333          & 3600 &  & 325       & \textbf{325}    & 401    \\
19                        & 796                 & 592   & 703          & 3600 &  & 700       & \textbf{700}    & 220    \\
20                        & 798                  & 550   & 648          & 3600 &  & 640       & \textbf{640}    & 547    \\
21                        & 323                 & 276.5 & \textbf{313} & 10   &  & 313       & \textbf{313}    & 466    \\
22                        & 516                  & 436   & \textbf{479} & 505  &  & 478.75    & \textbf{479}    & 1724   \\
23                        & 502                  & 396   & 459          & 3600 &  & 450       & \textbf{450}    & 777    \\
24                        & 850                  & 537   & 649          & 3600 &  & 590.5     & 592             & 3600   \\
25                        & 1046                 & 712   & 798          & 3600 &  & 791       & \textbf{791}    & 663    \\
26                        & 238                  & 166.5 & \textbf{192} & 173  &  & 192       & \textbf{192}    & 3255   \\
27                        & 326                  & 193   & 214          & 3600 &  & 213       & \textbf{213}    & 1267   \\
28                        & 372                  & 236   & 268          & 3600 &  & 261       & \textbf{261}    & 3411   \\
29                        & 718                  & 571   & 654          & 3600 &  & 647       & 648             & 3600   \\
30                        & 583                  & 316   & 417          & 3600 &  & 342       & \textbf{342}    &     2738   \\ \hline
\end{tabular}
\caption{Results for the Cplex and $B\&P$ on benchmark instances}
\label{TabResults}
\end{table}

The objective of the second experiment is to evaluate the benefit obtained by considering multiple units and assess the efficiency of the $B\&P$ algorithm to solve NRPMU instances. For there are no published solutions to tackle it, we compare our algorithm with a CIPSolver, Cplex. The integer programming model input into Cplex is obtained by extending the model from section 2.2 to consider a set of nurses. Cplex was executed with default settings and one hour time limit. The DP algorithm used by the $B\&P$ to conduct experiments is DPPI. Results are reported in Table \ref{TabResults}. For each instance and each method, the lower bound (LB), the upper bound (UB), and the computational time (Time) are provided. Note that the LB of Cplex stands for the linear relaxation value of corresponding integer programming model while the LB of the $B\&P$ represents the objective value of the MP at the root node of the branch and bound tree. The column Single reports the optimal objective obtained by solving each instance by assigning only nurses with preferred skill to corresponding units.
The \textbf{bold} values in the UB column represent that their optimality is proven by the corresponding optimizer.

As shown in Table 7, considering multiple units brings significant benefits to constructing high-quality rosters. According to our instances, the objective value in column UB of the $B\&P$ decreases by an average of $19.5\%$ compared to that in column Single.
Regarding solutions for the NRPMU, the $B\&P$ algorithm yields a lower bound that is tighter than or equal to that obtained by Cplex on all instances. For instances with two weeks, both Cplex and $B\&P$ succeeds in solving 14 instances out of 15 to optimality. The $B\&P$ algorithm consumes less computational time on 12 instances. For instances with four weeks, $B\&P$ is able to find 13 optimal solutions compared to 4 for Cplex.  $B\&P$ obtains better solutions on two instances that cannot be solved to optimality by either method. For instances 13 and 17, Cplex generates solutions with the optimal objective, but could not prove their optimality due to weak lower bounds. Apparently, with the increase in the number of nurses, units, and weeks, the proposed $B\&P$ algorithm  has demonstrated its potential to generate better solutions.

\section{Conclusion}
\label{sec_conclusion}
In this article, we study a new version of NRPs, NRPMU, which features multiple units and a number of soft time-related constraints. To solve it efficiently, an exact algorithm based on $B\&P$ is presented. How to tackle the pricing subproblem in the CG approach for the NRP has always been a challenging problem since various soft constraints are involved and existing DP algorithms suitable for VRPs cannot be directly applied. Based on ideas from CP constraints to model rostering problems, we define new labels and dominating rules specific to NRPMU and NRPs. Moreover, several strategies are developed to improve the performance of the proposed DP and $B\&P$ algorithms. Computational experiments performed on a variety of instances with different sizes demonstrate that the DP algorithm proposed in this study is efficient and the $B\&P$ algorithm used to solve the NRPMU is competitive.

 \bibliographystyle{elsarticle-harv}  
 \bibliography{cas-refs}





\end{document}